\documentclass[12pt,a4paper,reqno]{amsart}
\usepackage[utf8]{inputenc}
\usepackage[T1]{fontenc}
\usepackage{amsmath}
\usepackage{amsthm}
\usepackage{amssymb}
\usepackage[abbrev]{amsrefs}
\usepackage{mathrsfs}
\usepackage[dvipsnames]{xcolor}
\usepackage{bm}
\usepackage{enumitem}
\usepackage{hyperref}
\AtBeginDocument{\def\MR#1{}}
\makeatletter
\@namedef{subjclassname@2020}{%
\textup{2020} Mathematics Subject Classification}
\makeatother
\numberwithin{equation}{section}
\allowdisplaybreaks

\newtheorem{thm}{}[section]
\newtheorem{theorem}[thm]{Theorem}
\newtheorem{corollary}[thm]{Corollary}
\newtheorem{lemma}[thm]{Lemma}
\newtheorem{proposition}[thm]{Proposition}
\newtheorem{example}[thm]{Example}
\theoremstyle{definition}
\newtheorem{definition}[thm]{Definition}

\newcommand{\Id}{\ensuremath{\mathrm{Id}}}
\newcommand{\LL}{\ensuremath{\bm{L}}}
\newcommand{\YY}{\ensuremath{\mathbb{Y}}}

\newcommand{\XX}{\ensuremath{\mathbb{X}}}
\newcommand{\XB}{\ensuremath{\mathcal{X}}}
\newcommand{\YB}{\ensuremath{\mathcal{Y}}}
\newcommand{\Lt}{\ensuremath{\mathcal{L}}}
\newcommand{\ee}{\ensuremath{\bm {e}}}
\newcommand{\Ft}{\ensuremath{\mathcal{ F}}}
\newcommand{\Pt}{\ensuremath{\mathcal{P}}}
\newcommand{\St}{\ensuremath{\mathcal{S}}}
\newcommand{\Sym}{\ensuremath{\mathbb{S}}}
\newcommand{\bpsi}{\ensuremath{\bm{\psi}}}
\newcommand{\vPhi}{\ensuremath{\bm{\Phi}}}
\newcommand{\vPsi}{\ensuremath{\bm{\Psi}}}
\newcommand{\Fb}{\ensuremath{\bm{F}}}
\newcommand{\Gb}{\ensuremath{\bm{G}}}
\newcommand{\Dt}{\ensuremath{\mathcal{D}}}

\newcommand{\Kt}{\ensuremath{\mathcal{K}}}
\newcommand{\ct}{\ensuremath{\bm{C}}}
\newcommand{\kt}{\ensuremath{\bm{k}}}
\newcommand{\rt}{\ensuremath{\bm{r}}}
\newcommand{\Mt}{\ensuremath{\mathcal{M}}}
\newcommand{\Nt}{\ensuremath{\mathcal{N}}}
\newcommand{\Ot}{\ensuremath{\mathcal{O}}}
\newcommand{\Ct}{\ensuremath{\mathcal{C}}}
\newcommand{\pp}{\ensuremath{\bm{P}}}
\newcommand{\pq}{\ensuremath{\bm{Q}}}

\DeclareMathOperator{\rang}{rang}
\DeclareMathOperator{\supp}{supp}
\newcommand{\RR}{\ensuremath{\mathbb{R}}}
\newcommand{\FF}{\ensuremath{\mathbb{F}}}
\newcommand{\NN}{\ensuremath{\mathbb{N}}}

\newcommand{\abs}[1]{\left\lvert#1\right\rvert}
\newcommand{\norm}[1]{\left\lVert#1\right\rVert}

\newcommand{\enbrace}[1]{\left\lbrace#1\right\rbrace}

\newcommand{\enpar}[1]{\left(#1\right)}

\author[J. L. Ansorena]{Jos\'e L. Ansorena}\address{Department of Mathematics and Computer Sciences\\
Universidad de La Rioja\\
Logro\~no 26004\\ Spain}
\email{joseluis.ansorena@unirioja.es}
\author[G. Bello]{Glenier Bello}
\address{Departamento de Matem\'{a}ticas e
Instituto Universitario de Matem\'{a}ticas y Aplicaciones\\
Universidad de Zaragoza\\
50009 Zaragoza\\
Spain}
\email{gbello@unizar.es}
\begin{document}
\title[Basic sequences in variable exponent Lebesgue spaces]{On the basic sequence structure of variable exponent Lebesgue spaces}
\begin{abstract}
We study the subsymmetric basic sequence structure of variable exponent Lebesgue spaces $L_{\pp}$ built from index functions $\pp\colon\Omega\to(0,\infty]$ on $\sigma$-finite measure spaces $(\Omega,\Sigma,\mu)$. Specifically, we prove that if $\pp$ is bounded away from infinity, then any complemented subsymmetric basic sequence of $L_{\pp}$ is equivalent to the canonical basis of $\ell_r$ for some $r\ge 1$ in the essential range of $\pp$.
\end{abstract}
\thanks{Both authors acknowledge the support of the Spanish Ministry for Science and Innovation under Grant PID2022-138342NB-I00 for \emph{Functional Analysis Techniques in Approximation Theory and Applications (TAFPAA)}. G. Bello has also been partially supported by PID2022-137294NB-I00, DGI-FEDER and by Project E48\_23R, D.G. Arag\'{o}n}
\maketitle
\section{Introduction}\noindent
The theory of variable exponent Lebesgue spaces has matured from a topic of abstract interest into an indispensable tool in partial differential equations, fluid dynamics, and harmonic analysis. Unlike their classical counterparts, these spaces allow the exponent to vary with the spatial variable, providing a natural and powerful framework for modeling phenomena with localized behavior, such as electrorheological fluids or materials with inhomogeneous properties. The rich structure of these spaces is deeply intertwined with the properties of the exponent function. In this article, we regard these spaces from a functional analysis perspective.

Within the study of the geometry of Banach spaces, the structure of basic sequences is a fundamental topic already present in Banach's book \cite{Banach1932}. In these early years, specialists were confident of achieving structural results valid for all Banach spaces. In particular, they sought to prove that every Banach space contains a copy of $c_0$ or $\ell_p$ for some $1\le p<\infty$. Since the canonical bases of these spaces are symmetric, the study naturally evolved to determining the structure of symmetric basic sequences of Banach spaces. Any symmetric basis is subsymmetric, that is, unconditional and equivalent to all its subsequences. As a matter of fact, the only feature of symmetric sequences that one needs in many situations is their subsymmetry. Partly because of this, and partly because spreading models of null sequences are subsymmetric, the researchers went on to seek subsymmetric basic sequences.

The study of the basic sequence structure of Banach spaces is part of the study of their subspace structure or, following Banach \cite{Banach1932}, that of the `linear dimension' of Banach spaces. We point out that the mere existence of copy of a Banach space $\XX$ within another Banach space $\YY$ does not entitle us to claim that $\XX$ is in a block from which $\YY$ is constructed. To fairly assert this, we must assume that $\XX$ can be placed into $\YY$ in a complementary manner, so that $\YY$ is isomorphic to the direct sum of $\XX$ and a third Banach space. So, the study of the complemented basic sequence structure of Banach spaces is of special interest.

We refer the reader to \cite{AnsorenaBello2025b} for a summary of advances in understanding the structure of subsymmetric basic sequences of Banach spaces. Regarding Lebesgue spaces $L_p$, the first milestone within this topic is \cite{Paley1936}, where Paley solved some problems posed in \cite{Banach1932}. Namely, he proved that given $p$, $q\in[1,\infty)$, $\ell_q$ does not embed in $L_p$ unless $q=2$ or $p\le q\le 2$. The fact that $\ell_2$ embeds into $L_p$ is a ready consequence of Khintchine inequality \cite{Kinchine1923}. Kadets \cite{Kadec1958} completed the study by showing that if $p< q< 2$ then $\ell_q$ do embeds into $L_p$.

By duality, $\ell_q$ is not a complemented subspace of $L_p$ unless $p=1$ or $q\in\{2, p\}$. Clearly, $\ell_p$ is a complemented subspace of $L_p$ for any $1\le p<\infty$. Pe{\l}czy\'{n}ski \cite{Pel1960} proved that $\ell_2$ is a complemented subspace of $L_p$ for any $1<p<\infty$. Lindenstrauss and Pe{\l}czy\'{n}ski \cite{LinPel1968} completed the picture by proving that any unconditional basic sequence complemented in $L_1$ is equivalent to the $\ell_1$-basis.

Kadets and Pe{\l}czy\'{n}ski \cite{KadPel1962} carried out a systematic study of basic sequences in Lebesgue spaces. Among other results, they proved that, if $2<p<\infty$, the canonical bases of $\ell_2$ and $\ell_p$ are the only subsymmetric basic sequences of $L_p$. To the best of our knowledge, it is unknown whether $L_p$, $1\le p < 2$, contains symmetric or subsymmetric bases other than the $\ell_q$-basis, $p\le q \le 2$.

The paper that initiated the study of variable exponent Lebesgue spaces by basic sequence techniques was \cite{HR2012} (see also \cite{FHRS2020}). In it, the authors characterized in terms of the essential range of the variable exponent $\pp$ the indices $q\in[1,\infty)$ such that $\ell_q$ isomorphically embeds into a variable exponent Lebesgue $L_{\pp}$ over the unit interval. Hern\'andez and Ruiz also proved that $\ell_q$ complementably embeds into $L_{\pp}$ provided that $q$ belongs to the essential range of $\pp$.

In this paper, we complement the study carried out in \cite{HR2012} by characterizing the indices $q\in[1,\infty)$ such that $\ell_q$ complementably embeds into $L_{\pp}$. Besides, we widen the scope of the study by considering not only $\ell_q$-spaces but also subsymmetric basic sequences. Specifically, we characterize both the subsymmetric basic sequence structure and the complemented subsymmetric basic sequence structure of Lebesgue sapces with variable exponent essentially bounded. Oddly enough, we find out that, while symmetric basic sequences other than $\ell_q$-bases are contained in these spaces, they fail to be complemented unless they are $\ell_q$-bases. Unlike in \cites{HR2012,FHRS2020}, we include both nonatomic spaces and atomic ones in our study. The latter are also known as Bourgin--Nakano spaces (see \cites{Bourgin1943,Nakano1950}).

To conclude this introductory section, we outline the structure of the paper. In Section~\ref{sec:MO} we establish the notation about variable Lebesgue spaces that will be used through the paper. We see these spaces from the perspective of Musielak-Orlicz spaces, including locally nonconvex spaces. Without intending to tackle an exhaustive study of these spaces, we sketch the more relevant ideas to our purposes. We revisit some well-known results from our perspective, and also prove new ones. In Section~\ref{sec:Orlicz.functions}, we study certain Orlicz functions (those that arise as convex combinations of power functions) that naturally appear when studying the geometry of variable Lebesgue spaces. The study of basic sequences in these spaces is addressed in sections~\ref{sec:subsymmetric} and \ref{sec:complemented}, which contain the main results of the paper. While in Section~\ref{sec:complemented} we study complemented basic sequences, in Section~\ref{sec:subsymmetric} we give the results achieved without imposing complementability.
\section{Musielak-Orlicz function spaces}\label{sec:MO}\noindent
Given a measure space $(\Omega,\Sigma,\mu)$, we denote by $\Sigma(\mu)$ the set of all measurable sets with finite measure. Let $L(\mu)$ be the vector space of all measurable functions with values in the real or complex field $\FF$. As it is customary, we identify functions and sets that differ on a null set. We say that $\mu$ is separable if the metric space $(\Sigma(\mu),d_\mu)$ is separable, where $d_\mu$ is the distance given by
\[
d_\mu(A,B)=\mu(A \triangle B), \quad A,\, B\in\Sigma(\mu).
\]
We endow $L(\mu)$ with the topology associated with the convergence in measure on finite-measure sets, so $L(\mu)$ becomes a complete topological vector space. We denote by $\St(\mu)$ the vector space all simple integrable functions, i.e., the linear span of
\[
\enbrace{\chi_E \colon E\in\Sigma(\mu)}.
\]
A function space over $(\Omega,\Sigma,\mu)$ will be a quasi-Banach lattice $(\XX,\norm{\cdot}_\XX)$ such that
\begin{itemize}[leftmargin=*]
\item $\XX\subseteq L(\mu)$ continuously,
\item $\St(\mu)\subseteq \XX$, and
\item there is a constant $\ct\in[0,\infty)$ such that if a sequence $(f_n)_{n=1}^\infty$ in $\XX$ converges to $f$ $\mu$-a.e.\@ and $S:=\limsup_n \norm{f_n}_\XX<\infty$, then $f\in\XX$ and $\norm{f}_\XX\le \ct S$.
\end{itemize}
Let $L^+(\mu)$ be the cone of all measurable functions with values in $[0,\infty]$. Following \cite{AnsorenaBello2022}, a function quasi-norm will be a homogeneous map
\[
\rho\colon L^+(\mu) \to [0,\infty]
\]
such that
\begin{enumerate}[label=(F.\arabic*),widest=10,series=fqn,leftmargin=*]
\item\label{FQN:M} If $f\le g$ $\mu$-a.e., then $\rho(f)\le \rho(g)$,
\item\label{FQN:SubAdd} There is $\kt\in[1,\infty)$, such that, for all $f$, $g\in L^+(\mu)$,
\[
\rho(f+g)\le\kt \enpar{ \rho( f)+\rho( g)}.
\]
\item\label{FQN:Simple} $\rho( \chi_E)<\infty$ for all $E\in\Sigma(\mu)$,
\item\label{FQN:CM} for every $E\in\Sigma(\mu)$ and $\varepsilon>0$ there is $\delta>0$ such that $\mu(A)\le \varepsilon$ whenever $A\in\Sigma$ satisfies $A\subseteq E$ and $\rho(\chi_A)\le \delta$.
\item\label{FQN:RFatou} There is $\ct \in [1,\infty)$ such that $\rho(\lim_n f_n)\le \ct \lim_n \rho(f_n)$ for all non-decreasing sequences $(f_n)_{n=1}^\infty$ in $L^+(\mu)$.
\end{enumerate}
Let $0<r\le 1$. If the function quasi-norm $\rho$ satisfies
\begin{enumerate}[label=(F.\arabic*),widest=10,series=fqn,leftmargin=*,resume]
\item\label{FQN:SubPAdd} $\rho^r (f+g)\le \rho^r( f)+\rho^r( g)$ for all $f$, $g\in L^+(\mu)$
\end{enumerate}
we call it a function $r$-nom (function norm if $r=1$). Note that \ref{FQN:SubPAdd} implies \ref{FQN:SubAdd} with $\kt=2^{1/r-1}$.

Given a function quasi-norm $\rho$ the set
\[
\LL_{\rho}:=\enbrace{f\in L(\mu) \colon \norm{f}_\rho := \rho\enpar{\abs{f}}<\infty}
\]
endowed with the quasi-norm $\norm{\cdot}_\rho$ is a function space. The other way around, if $(\XX, \norm{\cdot}_\XX)$ is a function space then $\XX=\LL_{\rho}$, where $\rho$ is defined by $\rho(f)=\norm{f}_\XX$ if $f\in \XX$ and $\rho(f)=\infty$ otherwise.

Given $0<r<\infty$, a function quasi-norm $\rho$ and a family $\Ft=(f_j)_{j\in J}$ in $L^+(\mu)$ we set
\[
N_{\rho}(\Ft;r)=\rho\enpar{\enpar{\sum_{j\in J} f_j^r}^{1/r}}, \quad M_{\rho}(\Ft;r)=\enpar{\sum_{j\in J} \rho^r(f_j)}^{1/r}.
\]
We say that $\XX=\LL_\rho$ is lattice $r$-convex (resp., $r$-concave) if there is a constant $C$ such that $N_{\rho}(\Ft;r)\le C M_{\rho}(\Ft;r)$ (resp., $M_{\rho}(\Ft;r) \le C N_{\rho}(\Ft;r)$) for all families $\Ft$ in $L^+(\mu)$. If $\XX$ is $r$-convex (resp., $r$-concave) for some $r\in(0,\infty)$ we say that $\XX$ is $L$-convex (resp., $L$-concave). We say that $\XX$ is absolutely continuous if
\begin{enumerate}[label=(F.\arabic*),widest=6,series=fqn,leftmargin=*,resume]
\item\label{FQN:Abs} $\lim_n \rho(f_n)=0$ for all non-increasing sequences $(f_n)_{n=1}^\infty$ in $L^+(\mu)$ with $\rho(f_1)<\infty$ and $\lim_n f_n=0$.
\end{enumerate}
Let $\XX_0$ be the closure of $\St(\mu)$ in $\XX$. We say that $\XX$ is minimal if $\XX_0=\XX$. If $\XX$ is $L$-concave, then $\XX$ is absolutely continuous. In turn, if $\XX$ is absolutely continuous, then $\XX$ is minimal (see, e.g., \cite{AnsorenaBello2025b}).

Give a function quasi-norm $\rho$ over $(\Omega,\Sigma,\mu)$ we consider the gauge $\rho^*$ given by
\[
\rho^*(g)=\sup\enbrace{ \int_\Omega fg \, d\mu \colon f\in L^+(\mu), \, \rho(f)\le 1}, \quad g\in L^+(\mu).
\]
If
\begin{enumerate}[label=(F.\arabic*),widest=10,series=fqn,leftmargin=*,resume]
\item\label{FQN:L1loc}
for every $E\in\Sigma(\mu)$ there is $C_E\in(0,\infty)$ such that
\[
\int_E f\, d\mu\le C_E \rho(f), \quad f\in L^+(\mu),
\]
\end{enumerate}
then $\rho^*$ is a function norm. Note that \ref{FQN:L1loc} implies \ref{FQN:CM} and that $\rho^*$ always satisfies \ref{FQN:L1loc} as well. We say that $\rho^*$ is the dual function norm of $\rho$. The function space built from $\rho^*$ is called the conjugate space of $\XX=\LL_\rho$ and denoted by $\XX'$. We have $\XX''=\XX$ (see \cite{BennettSharpley1988}). We can regard $\XX'$ as a closed subspace of the dual space $\XX^*$. If $\XX$ is absolutely continuous, then $\XX'=\XX^*$ (see \cite{BennettSharpley1988}).

For $j=1$, $2$, let $\XX_j$ be a function space over a $\sigma$-finite measure space $(\Omega_j,\Sigma_j,\mu_j)$. Let
\[
K\colon \Omega_1\times \Omega_2 \to [0,\infty)
\]
be a measurable function. Suppose there is a bounded linear operator $T\colon \XX_1 \to \XX_2$ given by
\[
T(f)(\omega_2)=\int_{\Omega_1} K(\omega_1,\omega_2) f (\omega_1) \, d \mu_1(\omega_1), \quad f\in \XX_1, \, \omega_2\in \Omega_2.
\]
Then, there is a bounded linear operator $T'\colon \XX_2' \to \XX_1'$ given by
\[
T'(g)(\omega_1)=\int_{\Omega_2} K(\omega_1,\omega_2) g (\omega_2) \, d \mu_2(\omega_2), \quad g\in \XX_2, \, \omega_1\in \Omega_1.
\]
We say that $T$ is a kernel operator and that $T'$ is the adjoint operator of $T$.

Given a measurable function $f\colon \Omega \to G$, where $G$ is a topological commutative monoid (for instance $G=[0,\infty]$ or $G$ a vector space), we set
\[
\supp(f)=f^{-1}(0).
\]

Given a function space $\XX$ over a $\sigma$-finite measure space $(\Omega,\Sigma,\mu)$, and a pairwise disjointly supported sequence $\XB=(x_n)_{n=1}^\infty$ in $\XX\setminus\{0\}$, we set
\[
\Sym[\XX,\XB]=\enbrace{(a_n)_{n=1}^\infty \colon \sum_{n=1}^\infty a_n \, x_n \in \XX},
\]
and we endow $\Sym[\XX,\XB]$ with the topology it inherits from $\XX$.

An \emph{Orlicz function} will be a non-decreasing left-continuous function
\[
F\colon[0,\infty)\to[0,\infty]
\]
such that $\lim_{t\to 0^+} F(t)=0$ and $F(\infty):=\lim_{t\to\infty} F(t)>0$. Note that we allow Orlicz functions to take the value $\infty$. If it is the case for the Orlicz function $F$, then there is $0<c<\infty$ such that $F^{-1}(\infty)=(c,\infty)$.

Given a $\sigma$-finite measure space $(\Omega,\Sigma,\mu)$, a function
\[
M\colon\Omega\times [0,\infty)\to [0,\infty],
\]
and $t\in[0,\infty)$,
we denote by $\nu_M(\cdot,t)$ the measure on $(\Omega,\Sigma)$ given by
\[
\nu_M(A,t)=\int_A M(\omega,t)\, d\mu(\omega), \quad A\in\Sigma.
\]
We say that $M$ is a Musielak-Orlicz function (cf.\@ \cite{Musielak1983}*{Chapter~7}) if
\begin{itemize}[leftmargin=*]
\item $M(\omega,\cdot)$ is an Orlicz function for all $\omega\in\Omega$, and
\item for every $E\in\Sigma(\mu)$ there exists $t_E\in(0,\infty)$ such that $\nu_M(E,t_E)<\infty$ and $M(\omega,t_E)>0$ for all $\omega\in E$.
\end{itemize}
If $f\in L^+(\mu)$ and $M$ is a Musielak-Orlicz function, then $M(\cdot,f(\cdot))$ is a measurable function. Hence, associated with $M$ there is a gauge
\[
\rho_M\colon L^+(\mu)\to[0,\infty], \quad f\mapsto \int_\Omega M(\omega,f(\omega))\, d\mu(\omega).
\]
Consider the sets
\begin{align*}
L^+_M&=\enbrace{f\in L^+(\mu) \colon \rho_M(tf)<\infty \textnormal{ for some } t\in(0,\infty)},\\
L_M&=\enbrace{f\in L(\mu) \colon \abs{f} \in L_M^+},\\
H^+_M&=\enbrace{f\in L^+(\mu) \colon \rho_M(tf)<\infty \textnormal{ for all } t\in(0,\infty)}, \\
H_M&=\enbrace{f\in L(\mu) \colon \abs{f} \in H_M^+},
\end{align*}
and, given $\varepsilon>0$,
\[
B_M(\varepsilon)=\enbrace{f\in L(\mu) \colon \rho_M(\abs{f})<\varepsilon}.
\]

The gauge $\rho=\rho_M$ satisfies properties \ref{FQN:M}, \ref{FQN:CM}, \ref{FQN:RFatou} and \ref{FQN:Abs}, property~\ref{FQN:SubAdd:bis} below instead of \ref{FQN:SubAdd}, and property~\ref{FQN:Simple:bis} below instead of \ref{FQN:Simple}.
\begin{enumerate}[label=(F.\arabic*),widest=10,series=fqn, leftmargin=*,resume]
\item\label{FQN:SubAdd:bis} There are $\kt$ and $\rt$ in $(0,\infty)$ such that, for all $f$, $g\in L^+(\mu)$,
\[
\rho(f+g)\le\kt \enpar{ \rho( \rt f)+\rho(\rt g)}.
\]
\item\label{FQN:Simple:bis} For every $E\in\Sigma(\mu)$ there is $u_E\in(0,\infty)$ such that $\rho(u_E \chi_E)<\infty$.
\end{enumerate}
In fact, we can choose $\ct=1$ in \ref{FQN:RFatou}, $\kt=1$ and $\rt=2$ in \ref{FQN:SubAdd:bis}, and $u_E=t_E$ in \ref{FQN:Simple:bis}. Consequently,
\[
u B_M(\varepsilon), \quad u,\, \varepsilon\in(0,\infty),
\]
is neighbourhood basis at the origin for a complete vector topology on $L_M$. Besides, the $F$-space $L_M$ continuously embeds into $L(\mu)$. We say that $L_M$ is the Musielak-Orlicz space over $(\Omega,\Sigma,\mu)$ built from $M$. We denote by $L_M^0$ the closed subspace of $L_M$ generated by simple integrable functions.

By property~\ref{FQN:Abs}, the dominated convergence theorem holds in $L_M$ for sequences dominated by a function in $H_M^+$. We infer that $H_M$ is the closed linear span of
\[
\enbrace{\chi_E \colon E \in\Sigma(\mu), \chi_E\in H^+_{M}}.
\]
Hence, $H_M\subseteq L_M^0\subseteq L_M$, and $L_M^0=H_{M}$ if and only if $\chi_E\in H^+_{M}$ for all $E\in\Sigma(\mu)$. Consequently, $H_{M}=L_M^0$ provided that
\begin{enumerate}[label=(F.\arabic*),widest=10,series=fqn,leftmargin=*,resume]
\item\label{it:Doubling} there are $C$, $D\in(1,\infty)$ such that
\begin{equation*}
M(\omega,Dt)\le C M(\omega,t), \quad (\omega,t)\in\Omega\times[0,\infty).
\end{equation*}
\end{enumerate}

If
\begin{enumerate}[label=(F.\arabic*),widest=10,series=fqn,leftmargin=*,resume]
\item\label{it:QBL} there are $c$, $d\in(0,1)$ such that
\begin{equation*}
M(\omega,dt)\le c M(\omega,t), \quad (\omega,t)\in\Omega\times[0,\infty),
\end{equation*}
\end{enumerate}
then $L_M$ is a function space. Indeed, the Luxemburg map
\[
f\mapsto \rho^L_M(f):= \inf\enbrace{t\in(0,\infty) \colon \rho_M\enpar{f/t}<1}, \quad f\in L^+(\mu),
\]
besides being homogeneous, inherits from $\rho_M$ the properties \ref{FQN:M}, \ref{FQN:CM}, \ref{FQN:RFatou}, \ref{FQN:SubAdd:bis} and \ref{FQN:Simple:bis}. Hence, $\rho_M^L$ is a function quasi-norm such that $L_M=\LL_{\rho^L_M}$.

Let $0<r<\infty$ and $M$ be a Musielak-Orlicz function. If $M$ is $r$-convex, i.e., the mapping
\[
t \mapsto M\enpar{\omega, t^{1/r}}, \quad t\in[0,\infty),
\]
defines a convex function for all $\omega\in\Omega$, then \ref{it:QBL} holds, and the function space $L_M$ is lattice $r$-convex. Let $0<s<\infty$. If $M$ is $s$-concave, i.e., the mapping
\[
t\mapsto M\enpar{\omega, t^{1/s}}, \quad t\in[0,\infty),
\]
defines a concave function for all $\omega\in\Omega$, then it satisfies \ref{it:Doubling}. Therefore, if $M$ is $s$-concave and satisfies \ref{it:QBL}, then $L_M$ is a lattice $s$-concave an absolutely continuous function space, and $H_M=L_M$.

We will use the following two results whose straightforward proofs we omit.

\begin{lemma}\label{lem:AttainsNorm}
Let $M$ be a Musielak-Orlicz function satisfying \ref{it:QBL}. Let $f\in L_M\setminus\{0\}$ and $a=1/\norm{f}_M$. Then, either
\begin{enumerate}[label=(\alph*)]
\item\label{it:Attains} $\rho_M(a\abs{f})=1$ or
\item $\rho_M(a\abs{f})<1$ and $\rho_M (t\abs{f})=\infty$ for all $t\in(a,\infty)$.
\end{enumerate}
In particular, \ref{it:Attains} holds for all $f\in H_M$.
\end{lemma}

\begin{theorem}(see \cite{AnsorenaMarcos2025})\label{thm:MusContinuousEmbed}
Let $M$ and $N$ be Musielak-Orlicz functions on a $\sigma$-finite measure space $(\Omega,\Sigma,\mu)$. Suppose that there is $c\in(0,\infty)$ such that
\begin{itemize}
\item $\max\enbrace{\nu_M(\Omega,c),\nu_N(\Omega,c)}<\infty$,
\item $\min\enbrace{M(\omega,c), N(\omega,c)}>0$ for all $\omega\in\Omega$, and
\item $M(\omega,t)\le N(\omega,t)$ for all $\omega\in\Omega$ and $t\in[c,\infty)$.
\end{itemize}
Then $L_N(\mu)\subseteq L_M(\mu)$ continuously.
\end{theorem}

Next, we detail three types of Musielak-Orlicz spaces we will deal with.

\subsection{Musielak-Orlicz sequence spaces}
We regard Musielak-Orlicz functions over $\NN$, which we call Musielak-Orlicz sequences, as sequences of Orlicz functions. Given an Musielak-Orlicz sequence $\Fb=(F_n)_{n=1}^\infty$, we denote by $\ell_{\Fb}$ the Musielak-Orlicz space built from $\Fb$. Musielak-Orlicz sequence spaces naturally appear when studying the unconditional basic sequence structure of Musielak-Orlicz function spaces. Indeed, if $M$ is a Musielak-Orlicz function over a $\sigma$-finite measure space $(\Omega,\Sigma,\mu)$, and $\XB=(x_n)_{n=1}^\infty$ is a pairwise disjointly supported sequence in $L_M\setminus\{0\}$, then $\Sym[L_M,\XB]=\ell_{\Fb}$, where $\Fb=(F_n)_{n=1}^\infty$ is given by
\[
F_n(t)=\int_{\Omega} M\enpar{\omega, t \abs{x_n(\omega)}} \, d\mu(\omega), \quad n\in\NN.
\]

We bring up a result by Musielak that we will use.

\begin{theorem}[\cite{Musielak1983}*{Theorem 8.11}]\label{thm:MusEmbed}
Consider two Musielak-Orlicz sequences $\Fb=(F_n)_{n=1}^\infty$ and $\Gb=(G_n)_{n=1}^\infty$. Then, $\ell_{\Fb} \subseteq\ell_{\Gb}$ continuously if and only if there exist $\bm{a}=(a_n)_{n=1}^\infty\in\ell_1$ and $\delta,b$ and $C$ in $(0,\infty)$ such that
\[
G_n(t) \le a_n+ C F_n(bt)
\]
for all $n\in\NN$ and all $t\in(0,\infty)$ with $F_n(t)< \delta$.
\end{theorem}

\subsection{Orlicz sequence spaces}
Given an Orlicz function $F$ and a measure space $(\Omega,\Sigma,\mu)$, the Orlicz space $L_F(\mu)$ is the Musielak-Orlicz space built from the Musielak-Orlicz function $M$ over $(\Omega,\Sigma,\mu)$ given by $M(\omega,t)=F(t)$ for all $\omega\in\Omega$ and $t\in[0,\infty)$. If $\mu$ is the counting measure on $\NN$, we set $\ell_F=L_F(\mu)$, and we call $\ell_F$ the Orlicz sequence space built from $F$. Orlicz sequence spaces are Musielak-Orlicz sequence spaces built from constant Musielak-Orlicz sequences. Applying Theorem~\ref{thm:MusEmbed} in this particular situation yields the following result.

\begin{corollary}
Given two Orlicz functions $F$ and $G$, then $\ell_G\subseteq \ell_F$ if and only if there are $b$, $c$ and $C$ in $(0,\infty)$ such that
\[
G(t)\le C F(bt), \quad 0\le t \le c.
\]
\end{corollary}
Given a convex Orlicz function $F$ we define $F^*\colon[0,\infty)\to[0,\infty]$ as the optimal function $G$ such that
\[
uv\le F(u)+G(v), \quad u,v\in[0,\infty].
\]
The dual function $F^*$ of $F$ is a convex Orlicz function, and the conjugate space of $L_F(\mu)$ is $L_{F^*}(\mu)$. We record some properties of the mapping $F\mapsto F^*$ that we will need.

Given $p\in(0,\infty]$ we consider the potential Orlicz functions
\[
\vPhi_p,\, \vPsi_p\colon[0,\infty)\to [0,\infty), \quad \vPhi_p(t)= t^p, \, \vPsi_p(t)= \frac{ t^p}{p},
\]
with the convention that $1^\infty=0$. If $p\ge 1$, so that $\vPhi_p$ and $\vPsi_p$ are convex, we denote by $p'$ its conjugate exponent defined by $1/p+1/p'=1$. Note that $\vPhi_\infty=\vPsi_\infty=\infty \chi_{(1,\infty)} $.

We record a couple of lemmas about Orlicz functions. The first of them is clear and well-known.

\begin{lemma}\label{lem:dualpot}
$(\vPsi_p)^*=\vPsi_{p'}$ for all $p\in[1,\infty]$.
\end{lemma}

\begin{lemma}\label{lem:OrliczBound}
Let $F$ be a convex Orlicz function. Let $c$, $C$, $A$ and $v$ in $(0,\infty)$ be such that
\[
uv \le C F(u) + A, \quad 0\le u \le c.
\]
Then, $F^*(w)\le A/C$, where
\[
w=\min{\enbrace{\frac{v}{C}, \frac{F(c)}{c}}}.
\]
\end{lemma}
\begin{proof}
On the one hand, if $u\in[0,c]$,
\[
uw\le \frac{uv}{C}\le F(u)+\frac{A}{C}.
\]
On the other hand, if $u\in[c,\infty)$, by convexity,
\[
u w \le u \frac{F(c)}{c} \le F(u) \le F(u)+ \frac{A}{C}.\qedhere
\]
\end{proof}

\subsection{Variable exponent Lebesgue spaces}
A variable exponent on a measure space $(\Omega,\Sigma,\mu)$ will be a measurable function
\[
\pp\colon \Omega\to (0,\infty].
\]
The variable exponent Lebesgue space $L_{\pp}$ associated with $\pp$ will be the $F$-space over $\Omega$ built from the Musielak-Orlicz function
\[
(\omega,t)\mapsto M_{\pp}(\omega,t)= \vPhi_{\pp(\omega)}(t), \quad (\omega,t)\in\Omega\times[0,\infty),
\]
that is $L_{\pp}=L_{M_{\pp}}$. We denote by $\rho_{\pp}:=\rho_{M_{\pp}}$ the modular constucted from $M_{\pp}$. If the implicit measure $\mu$ over $(\Omega,\Sigma)$ we are considering could be in doubt, we put $L_{\pp}=L_{\pp}(\mu)$. Set also $H_{\pp}=H_{M_{\pp}}$.

Recall that the support of a Borel measure $\nu$ on a Lindel\"of space $X$, denoted by $\supp(\nu)$, is the smallest closed set $K$ such that $\nu(X\setminus K)=0$. Let
\[
R(\pp):=\supp(\pp(\mu))\subseteq[0,\infty]
\]
be the essential range of $\pp$, that is the set of all $p\in[0,\infty]$ such that
\[
\mu\enpar{V\cap \rang(\pp)}>0
\]
for all measurable neighbourhoods $V$ of $p$. Set
\[
\pp^-:=\min\enpar{R(\pp)}, \quad \pp^+=\max\enpar{R(\pp)}.
\]
Suppose that $\pp^->0$. Then, $M_{\pp}$ is $\pp^-$-convex, whence $L_{\pp}$ is a lattice $\pp^-$-convex function space over $(\Omega,\Sigma,\mu)$. We denote by $\rho^L_{\pp}$ the function quasi-norm constructed from $\rho_{\pp}$.

We will also consider the Musielak-Orlicz function over $(\Omega,\Sigma,\mu)$ given by
\[
(\omega,t) \mapsto M^a_{\pp}(\omega,t):= \vPsi_{\pp(\omega)}(t).
\]
Let $\rho^a_{\pp}$ be the modular associated with $M^a_{\pp}$. Let $\rho_{\pp}^{L,a}$ be the function quasi-norm constructed from $\rho^a_{\pp}$. Notice that
\[
\pp^{-} \rho^{L,a}_{\pp} \le \rho^L_{\pp} \le \sup\enbrace{p^{1/p} \colon \pp^{-}\le p <\infty} \rho^{L,a}_{\pp}
\]
Since $p^{1/p}$ is bounded away from zero at infinity as long as $p$ is bounded away from zero, the function quasi-norms $\rho^L_{\pp}$ and $\rho_{\pp}^{L,a}$ are equivalent. Both $\rho^L_{\pp}$ and $\rho^{L,a}_{\pp}$ are function $\pp^c$-norms where
\[
\pp^c=\min\{1,\pp^{-}\}.
\]
We denote by $\norm{\cdot}_{\pp}$ the $\pp^c$-norm for $L_{\pp}$ given by $\norm{f}_{\pp}=\rho_{\pp}^L\enpar{\abs{f}}$ for all $f\in L_{\pp}$.

When dealing with variable exponent Lebesgue spaces we must be aware that the identity
\[
\vPsi_p(st) = \vPhi_p(s) \vPsi_p(t)
\]
breaks down when $p=\infty$ and $0\le s <1 <t <\infty$. So, we must pay attention to the set
\[
\Omega^{\pp}_\infty:=\pp^{-1}(\infty).
\]
In fact, for any $f\in L^+(\mu)$, if $\norm{\cdot}_{\infty}$ denotes the norm of $L_\infty(\mu|_{\Omega^{\pp}_\infty})$,
\[
\rho^L_{\pp}(f)=\max\enbrace{ \norm{f|_{\Omega^{\pp}_\infty}}_{\infty} , \inf \enbrace{ t>0 \colon
\int_{\Omega\setminus\Omega^{\pp}_\infty} \frac{{f(\omega)}^{\pp(\omega)}}{t^{\pp(\omega)}} d\mu(\omega)<1}}.
\]

If $\pp^+<\infty$, then $M_{\pp}$ is $\pp^+$-concave. Consequently, if $\pp^->0$ and $\pp^+<\infty$, then $L_{\pp}$ is lattice $\pp^+$-concave, whence $H_{\pp}=L_{\pp}$ and $L_{\pp}$ is absolutely continuous.

If $\mu$ is a finite measure and $\pp$ and $\pq$ are variable exponents over $\mu$ with $\pp\le\pq$, then $L_{\pq} \subseteq L_{\pp}$ continuously by Theorem~\ref{thm:MusContinuousEmbed}.

Given a measurable function $h\colon\Omega\to(0,\infty)$, the pointwise multiplier
\[
f\mapsto f h
\]
defines an isomorphism from $L_{\pp}(\nu)$ onto $L_{\pp}(\mu)$, where
\[
d \nu = h_{\pp}\, d\mu, \quad h_{\pp}=h^{\pp} \chi_{\Omega\setminus\Omega^{\pp}_\infty} + h \chi_{\Omega^{\pp}_\infty}.
\]
Since we can choose $h$ so that $\int_\Omega h_{\pp} \, d\mu=1$, any variable exponent Lebesgue space is isomorphic to a variable exponent Lebesgue space over a probabilty space. If this probability space is nonatomic and separable then, by \cite{Caratheodory1939}, the corresponding variable exponent Lebesgue space is isomorphic to a variable exponent Lebesgue space over $[0,1]$.

To understand the structure of a variable exponent Lebesgue space we must take into account that if $(\Omega_j)_{j=1}^2$ is a partition of the measure space $(\Omega,\Sigma,\mu)$ then, for any variable exponent $\pp\colon\Omega \to(0,\infty]$, setting $\pp_j=\pp|_{\Omega_j}$ and $\mu_j=\mu|_{\Omega_j}$ for $j=1$, $2$,
\[
L_{\pp} \simeq L_{\pp_1} \oplus L_{\pp_2}.
\]
If $J$ is a countable set and $\pq\colon J\to(0,\infty]$ is a function, we denote by $\ell_{\pq}$ the variable exponent Lebesgue space over the counting measure on $J$ associated to the variable exponent $\pq$. These discrete variable exponent Lebesgue space were introduced by Nakano \cite{Nakano1950} in the case when $\pq^{-}\ge 1$, and by Bourgin \cite{Bourgin1943} in the case when $\pq^{+}\le 1$. In general, we call $\ell_{\pq}$ a Bourgin-Nakano space.

Splitting $\sigma$-finite measures into its purely atomic part and its nonatomic part, we infer that any variable exponent Lebesgue space is isomorphic to
\[
L_{\pp}\oplus \ell_{\pq}
\]
for some nonatomic probability measure space $(\Omega,\Sigma,\mu)$, some countable set $J$ and some variable exponents $\pp\colon\Omega\to(0,\infty]$ and $\pq\colon J \to(0,\infty]$. Summing up, any variable exponent Lebesgue space over any separable $\sigma$-finite measure space is isomorphic to $L_{\pp}\oplus \ell_{\pq}$ for some countable set $J$ and some variable exponents $\pp\colon[0,1]\to(0,\infty]$ and $\pq\colon J\to (0,\infty]$.

The dual space of $L_{\pp}$ only depends on the behaviour of $\pp$ on
\[
\Omega^{\pp}_c=\pp^{-1}([1,\infty]).
\]
Indeed, we have the following.
\begin{proposition}\label{prop:nulldual}
Let $\pp\colon\Omega\to(0,1)$ be a variable exponent on a nonatomic $\sigma$-finite measure space $(\Omega,\Sigma,\mu)$. Then, the dual space of $L_{\pp}$ is null.
\end{proposition}

\begin{proof}
Set $\Omega_{p,A}=\pp^{-1}((0,p]) \cap A$ for each $p\in(0,1)$ and $A\in\Sigma(\mu)$. Since
\[
\bigcup_{\substack{0<p<1\\ A\in \Sigma(\mu)}} \enbrace{f\in L_{\pp} \colon \supp(f)\subseteq \Omega_{p,A}}
\]
is dense in $L_{\pp}$, it suffices to prove the result in the case when $\mu$ is finite and $\pp(\Omega)\subseteq(0,p]$ for some $0<p<1$. In this case, since $L_p(\mu)\subseteq L_{\pp}$ continuously and $\St(\mu)$ is dense in $L_{\pp}$, the result follows from the fact that $(L_p(\mu))^*=\{0\}$.
\end{proof}

We say that a variable exponent $\pp$ over a $\sigma$-finite mesure space $(\Omega,\Sigma,\mu)$ is convex if $\pp(\Omega)\subseteq[1,\infty]$. The conjugate variable exponent of $\pp$ is the convex variable exponent $\pq$ over $(\Omega,\Sigma,\mu)$ defined by
\[
\pq(\omega)=\enpar{\pp(\omega)}'= \frac{\pp(\omega)}{\pp(\omega)-1}.
\]
The well-known duality relation between $L_{\pp}$ and $L_{\pq}$ relies on the following lemma.

\begin{lemma}\label{lem:dualB}
Let $\pq$ be the conjugate variable exponent of a variable exponent $\pp$ over a $\sigma$-finite measure space $(\Omega,\Sigma,\mu)$. Let $g\in L^+(\mu)$ with $\rho^L_{\pq}(g)=1$ and $\supp(g)\subseteq \pp^{-1}([p,\infty])$ for some $p>1$. Set $f={g^{\pq-1}}$, with the convention that $0^0=0$. Then,
\[
\int_\Omega f g \, d\mu= \rho_{\pp} (f) = \rho_{\pq}(g)=1.
\]
\end{lemma}

\begin{proof}
Our conditions on $g$ yields $g\in H^{+}_{\pq}$, whence $\rho_{\pq}(g)=1$ by Lemma~\ref{lem:AttainsNorm}. Since $fg=f^{\pp}=g^{\pq}$, we are done.
\end{proof}

With the information we have gathered, proving that the conjugate space of $L_{\pp}$ is $L_{\pq}$ takes little effort. On the one hand, by Lemma~\ref{lem:dualpot},
\[
\int_\Omega fg\, d\mu \le 2, \quad f,g\in L^+(\mu), \, \max\enbrace{\rho^{L,a}_{\pp}(f) , \rho^{L,a}_{\pq} (g) }\le 1.
\]
Therefore, the dual function norm $\rho^*$ of $\rho_{\pq}^L$ satisfies $\rho^*\le 2 \rho_{\pq}^{L,a}$. On the other hand, by Lemma~\ref{lem:dualB} and homogeneity, $\rho^*(g)\ge \rho^L_{\pq}(g)$ for all $g\in L^+(\mu)$ with $\supp(g)\subseteq \pp^{-1}([p,\infty])$ for some $p>1$. By the Fatou property, this inequality extends to any $g\in L^+(\mu)$ with $\supp(g)\subseteq \pp^{-1}((1,\infty])$. Since the conjugate space of $L_1(\Omega^{\pq}_\infty)$ is $L_\infty(\Omega^{\pq}_\infty)$, the inequality extends to any $g\in L^+(\mu)$.
\section{Orlicz functions generated near zero by power functions}\label{sec:Orlicz.functions}\noindent
Given a signed measure $\mu$ on $(0,\infty]$ we define
\[
\bpsi_\mu\colon [0,1]\to \RR,\quad t \mapsto \int_{(0,\infty]} \vPsi_p(t) \, d\mu(p).
\]
Let $\delta_p$ be the Dirac measure on $p\in (0,\infty]$. We have
\[
\bpsi_p:=\bpsi_{\delta_p}=\vPsi_p|_{[0,1]}.
\]

We will use several times the elementary inequality

\begin{equation}\label{eq:ts}
\frac{ v^s-u^s}{s} \le \frac{v^r-u^r}{r} , \quad 0<r\le s\le\infty, \, 0\le u\le v \le 1.
\end{equation}

\begin{lemma}\label{lem:PotCont}
The mapping $p\mapsto H(p):= \bpsi_p$ defines a nonincreasing continuous function from $(0,\infty]$ into the Banach lattice $\Ct([0,1])$.
\end{lemma}

\begin{proof}
Applying \eqref{eq:ts} with $u=0$ gives that $H$ is nonincreasing. This inequality also gives
\[
\norm{H(r)-H(s)}_\infty =\bpsi_r(1)-\bpsi_s(1)=\frac{1}{r}-\frac{1}{s}, \quad 0<r\le s\le\infty.\qedhere
\]
\end{proof}

\begin{definition}
Given $R\subseteq (0,\infty]$ nonempty, we denote by $\Kt(R)$ the the closed convex hull in $\Ct([0,1])$ of $\{\bpsi_p \colon p\in R\}$. We denote by $\Ot(R)$ the set of all Orlicz functions $F$ for which there are $0<c\le 1$ and $\varphi\in\Kt(R)$ such that
\[
F(t)\approx \varphi(t), \quad 0\le t \le c.
\]
Put
\[
\Ot((0,\infty])=\bigcup_{p>0} \Ot([p,\infty]).
\]
Given $r\in (0,\infty]$ we set
\[
\Ot(R,r^+)=\bigcap_{s>r} \Ot(R\cap [r,s]).
\]
Finally, we put $\Ot(r)=\Ot(\{r\})$ and $\Ot(r^+)=\Ot([r,\infty],r^+)$.
\end{definition}

Note that $\Ot(\infty)=\Ot(\infty^+)$ consists of all Orlicz functions that are null at a neibourhood of the origin.

\begin{example}
Let $0<r<\infty$ and $a\in\RR$. Define $F_{r,a}\colon[0,\infty)\to [0,\infty)$ by
\[
F_{r,a}(t)=\begin{cases} 0 & \mbox{ if } t=0, \\ t^r (-\log (t) )^{a} & \mbox{ if } 0 <t \le 1/e, \\ t^r& \mbox{ otherwise.} \end{cases}
\]
Suppose that $a>0$. Fix $s>r$ and $0<c<1$. If $0\le t \le c$, then
\[
\int_r^s \frac{t^p}{p} \frac{a \enpar{p-r}^{a-1}} { (s-r)^a }\, dp = C(t) t^r \enpar{- \log(t)}^a,
\]
where
\[
C(t)=\frac{a }{(s-r)^a} \int_0^{(r-s)/\log(t)} \frac{ v^{a-1} e^{-v} \, dv}{r-v \log(t)}.
\]
Since $C(t)$ is bounded away from zero and infinity, $F_{r,a} \in \Ot(r^+)$.
\end{example}

\begin{lemma}\label{lem:extension}
Let $R\subseteq (0,\infty]$ be closed, nonempty and bounded away from zero. Then, any function in $\Kt(R)$ extends to a function in $\Ot(R)$.
\end{lemma}

\begin{proof}
Just notice that any function $F\in\Kt(R)$ is nondecreasing, continuous, and satisfies $F(0)=0$.
\end{proof}

Given $f\colon[0,1] \to \RR$ and $p\in (0,\infty]$ we consider the function
\[
\pi_p(f)\colon[0,1] \to \RR, \quad \pi_p(f)(t)=f\enpar{t^p}.
\]

\begin{definition}
Let $ 0<r\le s\le \infty$. We denote by $\Kt_{r,s}$ the set consisting of all nondecreasing functions $\varphi\colon[0,1]\to[0,\infty)$ such that $\pi_r(\varphi)$ is convex, $\pi_s(\varphi)$ is concave,
\[
\frac{t^s-u^s}{s} \le \varphi(t)-\varphi(u)\le \frac{t^r-u^r}{r},
\quad 0\le u <t \le 1,
\]
and
\[
\varphi(bt) \le b^s \varphi(t), \quad 1<b<\infty, \, 0\le t \le \frac{1}{b}.
\]
\end{definition}

\begin{lemma}\label{lem:lemA}
Let $R\subseteq (0,\infty]$ be closed, nonempty and bounded away from zero. Set $r=\min(R)$ and $s=\max(R)$. Then $\Kt(R) \subseteq\Kt_{r,s}$.
\end{lemma}

\begin{proof}
We infer from inequality~\eqref{eq:ts} that $\bpsi_p\in \Kt_{r,s}$ for all $p\in R$. Since $\Kt_{r,s}$ is convex and closed, we are done.
\end{proof}

\begin{lemma}\label{lem:singleempty}
For $j=1$, $2$ let $R_j\subseteq (0,\infty]$ be closed, nonempty and bounded away from zero.
Set $r_1:=\max(R_1)$ and $r_2:=\min(R_2)$.
\begin{enumerate}[label=(\roman*),leftmargin=*]
\item\label{it:SE:a} If $r_1=r_2$, then $\Ot(R_1)\cap \Ot(R_2) =\Ot(r_1)$.
\item If $r_1<r_2$, then $\Ot(R_1)\cap \Ot(R_2) =\emptyset$.
\end{enumerate}
\end{lemma}

\begin{proof}
Pick $F\in \Ot(R_1)\cap \Ot(R_2)$ and $r_1\le r \le s\le r_2$. There are $0<c\le 1$, $\varphi_j\in\Ot(R_j)$ and $C_j\in(0,\infty)$, $j=1$, $2$, such that
\[
\frac{\varphi_1(t)}{C_1} \le F(t)\le C_2 \varphi_2(t), \quad 0\le t \le c.
\]
By Lemma~\ref{lem:lemA},
\[
\frac{\bpsi_r(t)}{C_1} \le F(t)\le C_2 \bpsi_s(t) , \quad 0\le t \le c.
\]
Consequently, $r\ge s$.
\end{proof}

\begin{lemma}\label{lem:tellapart}
Let $0<q<r\le \infty$. Then, $\Ot(q^+)\cap\Ot(r^+)=\emptyset$.
\end{lemma}

\begin{proof}
Choose $q<s_1<r<s_2$. Since $\Ot(q^+)\cap\Ot(r^+)\subseteq\Ot([q,s_1]) \cap \Ot([r,s_2])$, the result follows from Lemma~\ref{lem:singleempty}.
\end{proof}

\begin{lemma}\label{lem:compact}
Let $R\subseteq (0,\infty]$ be closed, nonempty and bounded away from zero.
\begin{enumerate}[label=(\roman*), leftmargin=*, widest=ii]
\item $\varphi\in\Kt(R)$ if and only if there is a probability measure $\mu$ over $R$ such that $\varphi=\bpsi_\mu$.
\item $\Kt(R)$ is a compact set.
\end{enumerate}
\end{lemma}

\begin{proof}
Let identify the space $\Mt(R)$ of all signed measures over $R$ with the dual space of $\Ct(R)$. The set $\Pt(R)$ of all probability measures over $R$ is a convex subspace of $\Mt(R)$. Besides, $\Pt(R)$ is closed relative to the weak* topology, and
\[
\Dt(R)=\{\delta_p\colon p\in R\}
\]
is the set of extreme points of $\Pt(R)$.

By \eqref{eq:ts}, the map
\[
G\colon [0,1]\to \Ct(R), \quad G(t)(p)= \bpsi_p(t)
\]
is continuous. Consequently, the map
\[
D\colon [0,1] \times \Mt(R), \quad (t,\mu)\mapsto \int_R \bpsi_p(t) \, d\mu(p)
\]
is continuous. We infer that the map
\[
E\colon \Mt(R) \to \Ct([0,1]) , \quad \mu \mapsto D(\mu,\cdot)=\bpsi_\mu,
\]
besides linear, is continuous. Note that $E(\delta_p)=\bpsi_p$ for all $p\in R$ and that
\[
E(\Pt(R))=\Lt(R):=\enbrace{\bpsi_\mu \colon \mu\in\Pt(R)}.
\]
Therefore, $\Lt(R)$ is compact by the Banach--Alaoglu Theorem, and the convex hull of $\{\bpsi_p\colon p\in R\}$ is dense in $\Lt(R)$ by the Krein--Milman Theorem. Consequently, $\Lt(R)=\Kt(R)$.
\end{proof}

\begin{lemma}\label{lem:minsupp}
Let $R\subseteq(0,\infty]$ be closed, nonempty, and bounded away from zero. Let $\mu\in\Pt(R)$ and $F$ be an Orlicz function with $F\approx \bpsi_\mu$.
If $r=\min(\supp(\mu))$, then $F\in\Ot(R,r^+)$.
\end{lemma}

\begin{proof}
Fix $s>r$. We have
\[
\lambda=\mu([r,s)) >0.
\]
Let $\nu$ be the probability measure obtained by restricting $\mu/\lambda$ to $[r,s)$. Let $0<t\le 1$. By Lemma~\ref{lem:PotCont},
\[
\bpsi_\mu(t)\le \bpsi_\nu(t) \le \frac{1}{\lambda} \bpsi_\mu(t).
\]
Since $\nu\in\Pt([R\cap[r,s])$, we are done.
\end{proof}

\begin{corollary}\label{cor:doubling}
Let $R\subseteq(0,\infty]$ be closed, nonempty, and bounded away from zero. Let $F\in\Ot(R)$ and $b\in(0,\infty)$. Then, there is $c\in(0,\infty)$ such that
$F(bt)\approx F(t)$ for $0\le t \le c$.
\end{corollary}

\begin{proof}
If $F\in\Ot(\infty)$ the result is clear. Otherwise, there are $0<r<s<\infty$ such that $F\in\Ot([r,s])$. Applying Lemma~\ref{lem:lemA} puts an end to the proof.
\end{proof}

\begin{lemma}\label{lem:NoPotencial}
Suppose that a decreasing sequence $(r_n)_{n=1}^\infty$ in $(0,\infty)$ converges to $r>0$. Set $R=\{r\}\cup\enbrace{r_n \colon n\in \NN}$. Let $(b_n)_{n=1}^\infty$ in $(0,\infty)$ be such that $\sum_{n=1}^\infty b_n=1$. Define $F\colon [0,\infty)\to [0,\infty]$ by
\[
F(t)=\sum_{n=1}^\infty t^{r_n} b_n.
\]
Then, $F\in \Ot(R,r^+)\setminus \Ot(r)$.
\end{lemma}

\begin{proof}
By Lemma~\ref{lem:minsupp} $F\in\Ot(p^+)$. By the dominated convergence theorem,
\[
\lim_{t\to 0^+} \frac{F(t)}{t^p}=0.
\]
Hence, $F\notin \Ot(p)$.
\end{proof}

Recall that $r\in[0,\infty]$ is an accumulation point of a set $R\subseteq[0,\infty]$ if $R\cap (V\setminus\{r\})\not=\emptyset$ for every neighbourhood $V$ of $r$. If $R\cap(r,s)\not=\emptyset$ for all $s>r$, we say that $r$ is a right-side accumulation point of $R$.

\begin{proposition}\label{prop:lefttendence}
Let $R\subseteq (0,\infty]$ be closed, nonempty, and bounded away from zero. Given $r\in R$, $\Ot(R,r^+)\setminus\Ot(r)$ is nonempty if and only if $r$ is a right-side accumulation point of $R$.
\end{proposition}

\begin{proof}
Let $F\notin\Ot(r)$ be such that $F\in\Ot(R\cap [r,s])$ for all $s>r$. Then, $\{r\}\subsetneq R\cap [r,s]$ for all $s>r$.

Suppose that $r$ is a right-side accumulation point of $R$. Let $(r_n)_{n=1}^\infty$ in $R$ be decreasing to $r$. Pick $(b_n)_{n=1}^\infty$ with $\sum_{n=1}^\infty b_n=1$. If $F$ is as in Lemma~\ref{lem:NoPotencial}, $F\in \Ot(R\cap [r,s])$ for all $s>r$.
\end{proof}

\begin{corollary}
Let $R\subseteq (0,\infty]$ be closed, nonempty, and bounded away from zero. The following are equivalent.
\begin{itemize}
\item $\Ot(R)=\cup_{r\in R} \Ot(r)$.
\item $R$ has no right-side acumulation point.
\end{itemize}
\end{corollary}

\begin{proof}
Just combine Proposition~\ref{prop:lefttendence} with Lemma~\ref{lem:tellapart}.
\end{proof}

\begin{proposition}\label{prop:OrliczGoesToExtreme}
Let $R\subseteq (0,\infty]$ be closed, nonempty, and bounded away from zero. Then,
\[
\Ot(R,r^+), \quad r\in R,
\]
is a partition of $\Ot(R)$.
\end{proposition}

\begin{proof}
Just combine Lemma~\ref{lem:tellapart}, Lemma~\ref{lem:compact} and Lemma~\ref{lem:minsupp}.
\end{proof}

\begin{corollary}
$\Ot(r^{-}) :=\cap_{s<r} \Ot([s,r]) =\Ot(r)$ for all $r\in(0,\infty]$.
\end{corollary}

\begin{proof}
Let $F\in \Ot(r^{-})$. By Lemma~\ref{prop:OrliczGoesToExtreme} for each $s<r$ there is $p(s)\in [s,r]$ such that $F\in\Ot(p(s)^+)$. By Lemma~\ref{lem:tellapart}, $p(s)=r$ for all $s<r$. Applying Lemma~\ref{lem:singleempty}\ref{it:SE:a} puts an end to the proof.
\end{proof}

\begin{corollary}\label{cor:34}
Let $q\in(0,\infty]$. Let $R\subseteq (0,\infty]$ be closed, nonempty, and bounded away from zero. Then $\vPsi_q\in\Ot(R)$ if and only if $q\in R$.
\end{corollary}

\begin{proof}
Suppose that $\vPsi_q\in\Ot(R)$. By Proposition~\ref{prop:OrliczGoesToExtreme}, there is $r\in R$ such that $\vPsi_q\in\Ot(r^+)$. By Lemma~\ref{lem:tellapart}, $q=r$.
\end{proof}

\begin{example}
Let $0<p<\infty$ and $a<0$. Then, $F_{p,a}$ does not belong to $\Ot(R)$ for any $R\subseteq (0,\infty]$. Otherwise, by Proposition~\ref{prop:OrliczGoesToExtreme} and Lemma~\ref{lem:lemA}, there would be $r\in (0,\infty)$ and $(C_s)_{s\ge r}$ in $(0,\infty)$ such that
\[
\frac{1}{C_s} \bpsi_s(t) \le F_{p,a}(t) \le C_r \bpsi_r(t), \quad s\in(r,\infty], \, t\in[0,1].
\]
The right-hand inequality would yield $p>r$. Then, applying the left-hand inequality with $r<s<p$ we would reach an absurdity.
\end{example}

\begin{theorem}\label{thm:DualOR}
Let $R\subseteq (0,\infty]$ be closed, nonempty, and bounded away from zero.
Let $F$ be a convex Orlicz function such that $F\in\Ot(R)$ and $F^* \in \Ot((0,\infty])$. Then, there is $r\in R \cap [1,\infty]$ such that $F\in\Ot(r)$.
\end{theorem}

\begin{proof}
By Proposition~\ref{prop:OrliczGoesToExtreme}, there are $r\in R$ and $q\in (0,\infty]$ such that $F\in\Ot(r^+)$ and $F^*\in\Ot(q^+)$. Since $F$ and $F^*$ are convex, there are $0<c<1$ and $0<C<\infty$ such that $\max\{F(t), F^*(t)\} \le Ct$ for all $t\in[0,c]$. By Lemma~\ref{lem:lemA}, $r$, $q\in [1,\infty]$.

By Lemma~\ref{lem:lemA}, there are $0<c_1<1$ and $0<C_1<\infty$ such that $F^*(t) \le C_1 \vPsi_q(t)$ for all $0 \le t \le c_1$. Hence,
\[
u v \le F(u) + F^*(v) \le F(u) + C_1 \vPsi_q(v), \quad 0 \le v \le c_1.
\]
By Lemma~\ref{lem:OrliczBound}, Lemma~\ref{lem:dualpot} and Corollary~\ref{cor:doubling} , there are $0<c_2<1$ and $0<C_2<\infty$ such that
$\vPsi_{q'}(t) \le C_2 F(t)$ for all $t\in[0,c_2]$. By Lemma~\ref{lem:lemA}, $r\le q'$.

Pick $s>q$. By Lemma~\ref{lem:lemA}, there are $0<c_3<1$ and $0<C_3<\infty$ such that $\vPsi_s(t)\le C_3 F^*(t)$ for all $0 \le t \le c_3$. By Lemma~\ref{lem:dualpot},
\[
u v \le \vPsi_{s'}(u) + \vPsi_{s}(v) \le \vPsi_{s'}(u) + C_3 F^*(v), \quad 0\le v \le c_3.
\]
By Lemma~\ref{lem:OrliczBound} and Corollary~\ref{cor:doubling}, there are $0<c_4<1$ and $0<C_4<\infty$ such that $F(t) \le C_4 \vPsi_{s'}(t)$ for all $t\in[0,c_4]$.
By Lemma~\ref{lem:lemA}, $s'\le s_1$ for all $s_1>r$. Consequently, $q'\le r$.

We have proved that $q'=r$ and that $\vPsi_{r}(t) \le C_2 F(t)$ for all $t\in[0,c_2]$. By Lemma~\ref{lem:lemA}, a reverse inequality holds, that is, there are $0<c_5<1$ and $0<C_5<\infty$ such that $F(t) \le C_5 \vPsi_{r}(t)$ for all $t\in[0,c_5]$.
\end{proof}
\section{Subsymmetric sequences in variable exponent Lebesgue spaces}\label{sec:subsymmetric}\noindent
A sequence space, say $(\Sym, \norm{\cdot}_\Sym)$, will be a function space over $\NN$ endowed with the counting measure. Let $(\ee_n)_{n=1}^\infty$ denote the unit vectors of $\FF^\NN$. If $\norm{\ee_n}_\Sym\approx 1$ for $n\in\NN$, we say that $\Sym$ is semi-normalized. If the mapping $T_\pi$ given by
\[
(a_n)_{n=1}^\infty \mapsto (a_{\pi(n)})_{n=1}^\infty
\]
defines an isometry of $\Sym$ for any permutation $\pi$ of $\NN$, we say the $\Sym$ is a symmetric sequence space. In turn, if $T_{\pi}$ defines a isometry from
\[
\{f\in\Sym\colon \supp(f)\subseteq\pi(\NN)\}
\]
onto $\Sym$ for every increasing map $\pi\colon\NN\to\NN$, we say that $\Sym$ is a subsymmetric sequence space. Any symmetric sequence space is subsymmetric.

We say that a sequence space $\Sym$ embeds into a function space $\XX$ if there is an isomorphic embedding $T\colon \Sym\to\XX$. If $(T(\ee_n))_{n=1}^\infty$ are pairwise disjointly supported, we say that $\Sym$ disjointly embeds into $\XX$. Note that the following are equivalent.
\begin{itemize}[leftmargin=*]
\item $\Sym$ disjointly embeds into $\XX$.
\item There is an isomorphic embedding $T_0\colon \Sym_0\to\XX$ such that $\enpar{T(\ee_n)}_{n=1}^\infty$ is disjointly supported, that is,
the unit vector system of $\Sym$ is equivalent to a disjointly supported sequence in $\XX_0$.
\item There is a pairwise disjointly supported sequence $\XB$ in $\XX\setminus\{0\}$ such that $\Sym=\Sym[\XX,\XB]$.
\end{itemize}

\begin{proposition}\label{prop:DSImpliesOrlicz}
Let $\pp$ be a variable exponent on a $\sigma$-finite measure space $(\Omega,\Sigma,\mu)$. Assume that $\pp^{-}>0$. Let $\XB=(x_n)_{n=1}^\infty$ be a disjointly supported semi-normalized sequence in $L_{\pp}$. Then, $\XB$ has a subsequence $\YB$ such that $\Sym[L_{\pp},\YB]=\ell_F$
for some $F\in \Ot(R(\pp))$.
\end{proposition}

\begin{proof}
We have $L_{\pp}[\XB]=\ell_{\Fb}$, where $\Fb=(F_n)_{n=1}^\infty$ is given by
\[
F_n(t)=\int_\Omega \vPsi_{\pp(\omega)} \enpar{t \abs{x_n(\omega)}} \, d\mu(\omega), \quad t\in[0,\infty).
\]
Set $b=\sup_n \rho_M^L\enpar{\abs{x_n}}$ and define, for each $n\in\NN$,
\[
G_n\colon[0,\infty)\to[0,\infty], \quad G_n(t)=\int_{\Omega\setminus\Omega^{\pp}_\infty} \frac{t^{\pp(\omega)}}{\pp(\omega)} \enpar{\frac{\abs{ x_n(\omega)}}{b}}^{\pp(\omega)} \, d\mu(\omega).
\]
Since $\norm{x_n|_{\Omega^{\pp}_\infty}}_\infty\le b$ for all $n\in\NN$, $F_n(t)=G_n(tb)$ for all $t\in[0,1/b]$. Besides
\[
c_n:= \int_{\Omega\setminus\Omega^{\pp}_\infty} \enpar{\frac{\abs{ x_n(\omega)}}{b}}^{\pp(\omega)} \, d\mu(\omega)\le 1, \quad n\in\NN.
\]
Consider the following dichotomy.
\begin{itemize}
\item $\limsup_n c_n>0$.
\item $\limsup_n c_n=0$.
\end{itemize}
In the former case, passing to a subsequence we can assume that $\inf_n c_n>0$. By Lemma~\ref{lem:compact} and Lemma~\ref{lem:extension}, passing to a further subsequence we can assume that
\[
\sum_{n=1}^\infty \sup_{0\le t \le 1}\abs{\frac{G_n(t)}{c_n}-F(t)}<\infty.
\]
By Lemma~\ref{thm:MusEmbed}, $\ell_{\Fb}=\ell_F$.

In the latter case, passing to a subsequence we can assume that $\sum_{n=1}^\infty c_n<\infty$. Since
\[
G_n(t) \le \frac{c_n}{\pp^-},
\quad 0\le t\le 1, \, n\in\NN,
\]
$\ell_{\Fb}=\ell_\infty$ by Lemma~\ref{thm:MusEmbed}. Assume by contradiction that $\infty\notin R(\pp)$. Then, $\pp^{+}<\infty$ and $\Omega^{\pp}_\infty$ is a null set. If $0<a<\inf_n \rho_M^L\enpar{\abs{x_n}}$,
\[
1
\le \int_\Omega \enpar{\frac{\abs{ x_n(\omega)}}{a}}^{\pp(\omega)} \, d\mu(\omega)
=\enpar{ \frac{b}{a}}^{\pp^{+}} c_n, \quad n\in \NN.
\]
We reach an absurdity, and the proof is over.
\end{proof}

The following result generalizes \cite{HR2012}*{Proposition~4.5}.

\begin{lemma}\label{lem:disjointsets}
Let $\pp$ be a variable exponent over a $\sigma$-finite nonatomic measure space $(\Omega,\Sigma,\mu)$. For each $n\in\NN$, let $A_n$ be a finite subset of $R(\pp)$. Set
\[
\Nt=\enbrace{(n,p) \colon (n,p)\in \NN\times (0,\infty], \, p\in A_n}.
\]
For each $(n,p)\in\Nt$, let $U_{n,p}$ be a neighbourhood of $p$. Then, there is a family $(\Omega_{n,p})_{(n,p)\in\Nt}$ of pairwise disjoint measurable sets such that $\mu(\Omega_{n,p})>0$ and $\pp(\Omega_{n,p})\subseteq U_{n,p}$ for all $(n,p)\in\Nt$.
\end{lemma}

\begin{proof}
Assume without loss of generality that $(A_n)_{n=1}^\infty$ is nondecreasing. Set $A_0=\emptyset$, $A=\cup_{n=1}^\infty A_n$,
\[
B=\enbrace{p\in A \colon \pp(\mu)\enpar{\enbrace {p}}=0}, \quad \Mt=\enbrace{(n,p)\in\Nt \colon p\in B}.
\]
Let $n_0$ be the smallest $n\in\NN$ such that $B_n:=A_{n}\cap B\not=\emptyset$. We recursively construct for each $n\in\NN$ a pairwise disjoint family
\[
(V_{n,p})_{p\in B_n}
\]
such that $V_{n,p}\subseteq U_{n,p}$ is a measurable neighbourhood of $p$ for all $p\in B_n$, and the set $D_n:=\bigcup_{p\in B_n} V_{n,p}$ satisfies
\[
\pp(\mu) \enpar{D_n} <\frac{1}{2} \min_{p\in B_{n-1}} \pp(\mu) \enpar{V_ {n-1,p}}
\]
as long as $n\ge n_0+1$.
Set $E_n=\cup_{k=n}^\infty D_k$ for all $n\in\NN$. Since $\pp(\mu) \enpar{D_{n+1}}\le \pp(\mu) \enpar{D_{n}}/2$ for all $n\in\NN$, $n\ge n_0$, $\pp(\mu)(E_n)\le 2 \pp(\mu)(D_n)$ for all $n\in\NN$, $n\ge n_0$.
Consequently, $\pp(\mu)(E_{n+1})<\pp(\mu)\enpar{V_ {n,p}}$ for all $n\in\NN$ and $p\in B_n$. Set
\[
W_{n,p}=V_{n,p} \setminus E_{n+1}, \quad (n,p)\in \Mt.
\]
The family $(W_{n,p})_{(n,p)\in\Mt}$ is pairwise disjoint, and $\pp(\mu)(W_{n,p})>0$ for all $(n,p)\in\Mt$. Consider the set
\[
\Lt=\enbrace{(n,p)\in \Mt \colon W_{n,p}\cap (A\setminus B)}=\emptyset.
\]
Let $\pi\colon \Nt\setminus \Lt\to A\setminus B$ be such that $\pi(n,p)=p$ for all $(n,p)\in\Nt\setminus \Mt$, and $\pi(n,p)\in W_{n,p}$ for all $(n,p)\in\Mt\setminus \Lt$. Since $\mu$ is nonatomic, for each $q\in A\setminus B$ there is a pairwise disjoint family
\[
(\Omega_{n,p})_{(n,p)\in\pi^{-1}(q) }
\]
of nonnull measurable subsets of $\pp^{-1}(p)$. If for each $(n,p)\in\Lt$ we choose
\[
\Omega_{n,p}=\pp^{-1}(W_{n,p}),
\]
the family $(\Omega_{n,p})_{(n,p)\in\Nt}$ satisfies the desired conditions.
\end{proof}

\begin{corollary}\label{cor:subsymdisjoint}
Let $\pp$ be a variable exponent on a nonatomic $\sigma$-finite measure space $(\Omega,\Sigma,\mu)$. Assume that $\pp^{-}>0$. Let $\Sym$ be a subsymmetric sequence space. The following are equivalent.
\begin{itemize}
\item $\Sym$ disjointly embeds into $L_{\pp}$.
\item $\Sym=\ell_F$ for some $F\in\Ot(R(\pp))$.
\end{itemize}
\end{corollary}

\begin{proof}
Bearing in mind Proposition~\ref{prop:DSImpliesOrlicz}, it suffices to prove that for any Orlicz function $F$ such that $F|_{[0,1]}\in\Kt(R)$, $L_{\pp}^0$ has a disjointly supported sequence $\XB$ with $L_{\pp}[\XB]=\ell_F$. Choose $(\varepsilon_n)_{n=1}^\infty$ with $\sum_{n=1}^\infty \varepsilon_n<\infty$.
Let $(F_n)_{n=1}^\infty$ be such that, for all $n\in\NN$, $F_n$ belongs to convex hull of $\{F_p\colon p\in R(\pp)\}$, and
\[
\max_{0\le t\le 1} \abs{F(t)-F_n(t)}\le \varepsilon_n.
\]
Let $(A_n)_{n=1}^\infty$ be a sequence of finite subsets of $R(\pp)$ such that for each $n\in\NN$ there is
$(a_{n,p})_{p\in A_n}$ in $(0,\infty)$ such that
\[
F_n=\sum_{p\in A_n} a_{n,p} \vPsi_p, \quad \sum_{p\in A}a_{n,p}=1.
\]
By Lemma~\ref{lem:PotCont}, for each $n\in\NN$ and $p\in A_n$ there is a neibourghood $U_{n,p}$ of $p$ such that $\norm{\bpsi_p-\bpsi_q}\le \varepsilon_n$ for all $q\in U_{n,p}$. Let $(\Omega_{n,p})_{(n,p)\in\Nt}$ be the family of sets provided by Lemma~\ref{lem:disjointsets}. Choose, for each $(n,p)\in\Nt$, $f_{n,p}\colon \Omega \to [0,\infty)$ with $\supp(f_{n,p})=\Omega_{n,p}$ and $\int_\Omega f_{n,p}\, d\mu=1$. Set
\[
G_{n,p}\colon[0,\infty) \to[0,\infty), \quad t\mapsto \int_{\Omega} \vPsi_{\pp(\omega)}(t) f_{n,p}(\omega)\, d\mu(\omega).
\]
Note that, if $0\le t \le 1$,
\[
G_{n,p}(t)= \int_{\Omega\setminus \Omega^{\pp}_\infty} \vPsi_{\pp(\omega)}(t) f_{n,p}(\omega)\, d\mu(\omega).
\]
Since
\[
\max_{0\le t \le 1} \abs{G_{n,p}(t)-\bpsi_p(t)}\le \varepsilon_n,
\]
the Musielak-Orlicz sequence $\Gb=(G_n)_{n=1}^\infty$ given by
\[
G_n=\sum_{p\in A_n} a_{n,p} G_{n,p}
\]
satisfies $\max_{0\le t \le 1} \abs{G_n(t)-F(t)} \le 2 \varepsilon_n$ for all $n\in\NN$. By Theorem~\ref{thm:MusEmbed},
$\ell_F=\ell_{\Gb}$. If we set
\[
x_n=\sum_{p\in A_n}\enpar{ a_{n,p} f_{n,p}}^{1/\pp},
\]
then for every $n\in\NN$ and $0\le t \le 1$ we have
\[
\int_\Omega N_{\pp} \enpar{t x_n(\omega)} d\mu(\omega)=\sum_{p\in A_n} \int_{\Omega\setminus \Omega^{\pp}_\infty} \vPsi_{\pp(\omega)}(t) a_{n,p} f_{n,p}(\omega) \, d\mu(\omega)=G_n(t).
\]
Consequently, $\XB=(x_n)_{n=1}^\infty$ is a pairwise disjointly supported sequence with $L_{\pp}[\XB]=\ell_{\Gb}$.
\end{proof}

\begin{corollary}[see \cite{HR2012}*{Theorem~3.5}]
Let $q\in(0,\infty]$ and $\pp$ be a variable exponent over a $\sigma$-finite measure space.
Assume that $\pp^{-}>0$. Then, $L_{\pp}$ has a disjointly supported sequence equivalent to the unit vector system of $\ell_q$ if and only if $q\in R(\pp)$.
\end{corollary}

\begin{proof}
Just combine Corollary~\ref{cor:subsymdisjoint} with Corollary~\ref{cor:34}.
\end{proof}

Given a function space $\XX$, the following are equivalent.
\begin{itemize}
\item $\XX$ fails to be absolutely continuous.
\item $\ell_\infty$ embeds into $\XX$.
\item $\ell_\infty$ disjointly embeds into $\XX$.
\end{itemize}
Since any separable Banach space linearly embeds into $\ell_\infty$, the subsymmetric embedding problem for locally cxonvex K\"othe spaces split into two issues. Namely,
\begin{itemize}
\item determining which nonseparable subsymmetric spaces embed into a given non absolutely continuous K\"othe space, and
\item determining which separable subsymmetric spaces embed into a given absolutely continuous K\"othe space.
\end{itemize}
We give a result for variable exponent Lebesgue spaces framed in the latter problem.

\begin{theorem}\label{thm:MainA}
Let $\pp$ be a variable exponent over a $\sigma$-finite nonatomic measure space. Suppose that $\pp^{-}\ge 2$ and $q:=\pp^{+}<\infty$.
Let $\Sym$ be a subsymmetric sequence space. Then, $\Sym$ embeds into $L_{\pp}(\mu)$ if and only if $\Sym=\ell_2$ or $\Sym=\ell_F$ for some $F\in\Ot(R(\pp))$.
\end{theorem}

\begin{proof}
Assume that a symmetric sequence space $\Sym$ other than $\ell_2$ embeds into $L_{\pp}(\mu)$. By \cite{AnsorenaBello2025b}*{Theorem 4.2}, $L_{\pp}(\mu)$ contains a disjointly supported sequence equivalent to the canonical basis of $\Sym$. By Corollary~\ref{cor:subsymdisjoint}, $\Sym=\ell_F$ for some $F\in\Ot(R(\pp))$.

Corollary~\ref{cor:subsymdisjoint} also gives that $\ell_F$ embeds into $L_{\pp}(\mu)$ for any $F\in\Ot(R(\pp))$. We close the proof by noticing that since $L_2(\mu)\subseteq L_{\pp}(\mu)\subseteq L_q(\mu)$, any Rademacher sequence over $\Omega$ is equivalent to the canonical $\ell_2$-basis by Khintchine's inequalities.
\end{proof}

To address generalizing Theorem~\ref{thm:MainA} to the case when $\pp^{-}<2$, we should previously know the subsymmetric structure of $L_p$, $1\le p<2$.

The following elementary lemma implies that the situation in the atomic case is somehow different. Note that the space of simple integrable functions over $\NN$ endowed with the counting measure is the linear space $c_{00}$ of all eventually null sequences.

\begin{lemma}\label{lem:Dilworth}
Let $\XX$ be an absolutely continuous sequence space and $\Sym$ a subsymmetric sequence space. Then, $\Sym$ embeds into $\XX$ if and only if it disjointly embeds into $\XX$.
\end{lemma}

\begin{proof}
Assume that a sequence $(x_n)_{n=1}^\infty$ of $\XX$ is equivalent to the unit vector system on $\XX$. Applying the diagonal Cantor's technique we obtain an increasing sequence $(n_k)_{k=1}^\infty$ in $\NN$ such that $\enpar{x_{n_k}}_{k=1}^\infty$ converge pointwise. Consequently,
\[
\XB:=\enpar{x_{n_{2k-1}}-x_{n_{2k}}}_{k=1}^\infty
\]
converges to zero pointwise. Applying the gliding-hump technique, we infer that a subsequence of $\YB$ of $\XB$ is equivalent to a disjointly supported sequence of $\XX$. Since $\YB$ is equivalent to the unit vector system of $\Sym$, we are done.
\end{proof}

Given a sequence $\pp\colon \NN\to (0,\infty]$ we denote by $A(\pp)$ the set of all limit points of $\pp$. Note that any accumulation point of $R(\pp)$ is a limit point of $\pp$, but the converse does not hold. If $\pp=(p_n)_{n=1}^\infty$ and we set $\pp_m=(p_{n+m-1})_{n=1}^\infty$ for all $m\in\NN$, then
\[
A(\pp)= \bigcap_{m=1}^\infty R(\pp_m).
\]

\begin{theorem}
Let $\pp=(p_n)_{n=1}^\infty$ be a variable exponent with $\pp^{-}>0$ and $\pp^{+}<\infty$. Suppose a subsymmetric sequence space $\Sym$ embeds into $\ell_{\pp}$. Then, there is $r\in A(\pp)$ such that $\Sym=\ell_F$ for some Orlicz function $F$ such that $F\in\Ot(R(\pp)\cap[r,s])$ for all $s>r$.
\end{theorem}

\begin{proof}
Set $\pp_m=(p_{n+m-1})_{n=1}^\infty$ for all $m\in\NN$. By Lemma~\ref{lem:Dilworth}, Proposition~\ref{prop:DSImpliesOrlicz} and Proposition~\ref{prop:OrliczGoesToExtreme}, for each $m\in\NN$ there is $r_m\in R(\pp_m)$, and an Orlicz function $F_m$ such that $\Sym=\ell_{F_m}$ and $F_m\in \Ot(R(\pp_m) \cap [r_m,s])$ for all $s>r_m$. By Corollary~\ref{cor:doubling}, there is an Orlicz function $F$ such that $F_m\approx F$ for all $m\in\NN$. By Lemma~\ref{lem:tellapart}, there is $r\in(0,\infty]$ such that $r=r_m$ for all $m\in\NN$.
\end{proof}

\begin{example}
Let $\pp=(p_n)_{n=1}^\infty$ be a variable exponent for which there exists $\lim_n p_n=p\in(0,\infty)$. Then, $\ell_p$ is the unique subsymmetric space that embeds into $\ell_{\pp}$. Indeed, if a subsymmetric space $\Sym$ embeds into $\ell_p$, then, by Lemma~\ref{lem:Dilworth}, the unit vector system of $\Sym$ is equivalent to a disjointly finitely supported sequence $\XB=(x_j)_{j=1}^\infty$ in $\ell_{\pp}$. We can assume that $\XB$ is nonnegative and $\rho_{\pp}^{L,a}(x_j)=1$ for all $j\in\NN$. Passing to a subsequence we can assume that there is $(\varepsilon_j)_{j=1}^\infty$ in $(0,\infty)$ such that $\sum_{j=1}^\infty \varepsilon_j>0$ and
\[
\norm{\bpsi_q-\bpsi_p}_\infty \le \varepsilon_j, \quad q\in \supp(x_j).
\]
Since $\rho_{\pp}^a(x_j)=1$ for all $j\in\NN$,
\[
\abs{\bpsi_p-\sum_{n=1}^\infty M_{\pp}^a(n, tx_j(n))}\le \varepsilon_j, \quad 0\le t \le 1, \, j\in\NN.
\]
By Theorem~\ref{thm:MusEmbed} $\ell_{\pp}[\XB] =\ell_p$.
\end{example}

\begin{example}
Let $(q_j)_{j=1}^\infty$ be a sequence in $(0,\infty)$ decreasing to $q\in(0,\infty)$. Let $\sigma \colon \NN^2\to \NN$ be a bijection. Let $\beta\colon \NN\to\NN$ be the second component of the inverse of $\sigma$. Consider the variable exponent
$\pp=(p_n)_{n=1}^\infty$ given by
\[
p_n=q_{\beta(n)}.
\]
The Bourgin-Nakano space $\ell_{\pp}$ contains a symmetric space that is not an $\ell_p$-space. Indeed, if $(a_j)_{j=1}^\infty$ is a sequence in $(0,\infty)$ such that
\[
\sum_{n=1}^\infty a_j^{q_j}=1,
\]
the disjointly supported basic sequence $\XB=(x_k)_{k=1}^\infty$ defined by
\[
x_k=\sum_{j=1}^\infty a_j \ee_{\sigma(k,j)}.
\]
Since, for any $k\in\NN$ and $t\in[0,\infty)$,
\[
\sum_{n=1}^\infty \enpar{t x_k(n)}^{p_n}= F(t):=\sum_{j=1}^\infty a_j^{q_j} t^{q_j},
\]
$\ell_{\pp}[\XB]=\ell_F$. By Lemma~\ref{lem:NoPotencial}, $\ell_F\not=\ell_p$ for all $p\in(0,\infty]$.
\end{example}
\section{Complemented sequences in variable exponent Lebesgue spaces}\label{sec:complemented}\noindent
Let $\XX $ be a maximal K\"othe space over a $\sigma$-finite mesure space $(\Omega,\Sigma,\mu)$. We say that a sequence space $\Sym$ complementably embeds into $\XX$ if there are bounded linear maps $T\colon \Sym\to\XX$ and $P\colon\XX\to\Sym$ such that $P\circ T=\Id_{\Sym}$. If, $\Sym$ complementably embeds into $\XX$, then there are $\XB=(x_n)_{n=1}^\infty$ in $\XX$ and $\XB^*=(x_n^*)_{n=1}^\infty$ in $\XX^*$, called projecting functionals for $\XB$, such that the unit vector system of $\Sym$ is equivalent to $\XB$, $(\XB,\XB^*)$ is a biorthogonal system, and the mapping
\[
f\mapsto \sum_{n=1}^\infty x_n^*(f) x_n
\]
defines a bounded operator from $\XX$ into $\XX$. Besides, if $\Sym$ is absolutely continuous or $\XB$ is disjointly supported, then the converse also holds. If $\XB$ is disjointly supported and $\XB^*$ is a disjointly supported sequence in $\XX'$, we say that $\Sym$ disjointly embeds into $\XX$.

\begin{lemma}\label{lem:disjointcoordinate}
Let $\XX $ be a maximal K\"othe space and $\Sym$ be a sequence space. Suppose that $\XX$ is L-convex, that is, lattice $p$-convex for some $p>0$.
Suppose that there is a disjointly supported sequence $\XB=(x_n)_{n=1}^\infty$ in $\XB$ with projecting functionals $\XB^*$ in $\XX'$ such that $\Sym=\Sym[\XX,\XB]$. Then, $\Sym$ disjointly embeds into $\XX$.
\end{lemma}

\begin{proof}
Let $P$ be the endomorphism of $\XX$ given by
\[
f \mapsto \sum_{n=1}^\infty x_n^*(f) x_n.
\]
By \cite{Kalton1984b}, there is a constant $C$ such that
\[
\norm{\enpar{\sum_{j\in J} \abs{P(f_j)}^2}^{1/2}}\le C \norm{\sum_{j\in J} \abs{f_j}^2}^{1/2}
\]
for every family $(f_j)_{j\in J}$ in $\XX$. Pick a partition $(A_n)_{n=1}^\infty$ of $\XX$ so that $\supp(x_n)\subseteq A_n$ for all $n\in\NN$. Set $y_n^*=x_n^* \chi_{A_n}$ for all $n\in\NN$. Given $f\in\XX$, applying the above estimate to $\enpar{f\chi_{A_n}}_{n=1}^\infty$ we obtain
\[
\norm{\sum_{n=1}^\infty y_n^*(f) x_n} \le C \norm{f}.
\]
Consequently, $(y_n^*)_{n=1}^\infty$ are projecting functionals for $\XB$.
\end{proof}

\begin{lemma}\label{lem:Anso39}
Let $\pp$ be a convex variable exponent over a $\sigma$-finite measure space $(\Omega,\Sigma,\mu)$. Suppose that $\pp^{-}>0$ and that a subsymmetric sequence space $\Sym$ complementably disjointly embeds into $\LL_{\pp}$. Then, $\Sym=\ell_q$ for some $q\in R(\pp)$.
\end{lemma}

\begin{proof}
There are pairwise disjointly supported sequences $\XB=(x_n)_{n=1}^\infty$ in $\LL_{\pp}$ and $\XB'=(x_n')_{n=1}^\infty$ in $\LL_{\pp'}$ such that $\XB$ is equivalent to the unit vector system of $\Sym$ and $\XB'$ is equivalent to the unit vector system of $\Sym'$. By Proposition~\ref{prop:DSImpliesOrlicz}, there are $F\in\Ot(R(\pp))$ and $G\in\Ot(R(\pp'))$ such that $\Sym=\ell_F$ and $\Sym'=\ell_G$. We have $\ell_{F^*}=\ell_G$, whence, by Corollary~\ref{cor:doubling}, $F^*\approx G$ near the origin. Consequently, $F^*\in\Ot(R(\pp'))$. By Theorem~\ref{thm:DualOR}, $F\in\Ot(r)$ for some $r\in R(\pp)$.
\end{proof}

\begin{lemma}\label{lem:DropNLC}
Let $\pp$ be a variable exponent over a nonatomic $\sigma$-finite measure space $(\Omega,\Sigma,\mu)$. Suppose that $\pp^{-}>0$. Let $\Sym$ be a subsymmetric sequence space. Set $\pp_c=\pp|_{\Omega^{\pp}_c}$. Then, $\Sym$ complementably embeds into $L_{\pp}$ if and only if it complementably embeds into $L_{\pp_c}$.
\end{lemma}

\begin{proof}
Let $T\colon\Sym\to L_{\pp}$ and $P\colon L_{\pp}\to \Sym$ be such that $P\circ T=\Id_{\Sym}$. Let $J$ be the canonical embedding of $L_{\pp_c}$ into $\XX$, and $Q$ be the canonical projection from $\XX$ onto $L_{\pp_c}$. By Proposition~\ref{prop:nulldual}, $P=P\circ J\circ Q$. Consequently, $\Sym$ complementably embeds into $\YY$.
\end{proof}

\begin{theorem}
Let $\pp$ be a variable exponent over a nonatomic $\sigma$-finite measure space $(\Omega,\Sigma,\mu)$. Suppose that $\pp^{-}>0$ and $q:=\pp^{+}<\infty$ and that $\pp^{-1}((1,\infty))$ is not null. Let $\Sym$ be a subsymmetric sequence space. Set $\pp_c=\pp|_{\Omega^{\pp}_c}$. The following are equivalent.
\begin{itemize}
\item $\Sym$ complementably embeds into $L_{\pp}$.
\item $\Sym=\ell_2$ or $\Sym$ complementably disjointly embeds into $L_{\pp}$.
\item $\Sym=\ell_r$ for some $r\in \{2\} \cup R(\pp_c)$.
\end{itemize}
\end{theorem}

\begin{proof}
There is $p>1$ such that $\Omega_0:=\pp^{-1}([p,q])$ is not null. We can without loss of generality assume that $\mu$ is finite and $\mu(\Omega_0)=1$.

Let $(r_n)_{n=1}^\infty$ be a Rademacher sequence over $\Omega_0$. Let $S$ be the associated embedding of $\ell_2$ into $L_q(\Omega_0)$, and $Q$ the associated projection from $L_p(\Omega_0)$ onto $\ell_2$. Let $J$ the canonical map from $L_q(\Omega_0)$ into $L_{\pp}$, and $P$ be the canonical map from $L_{\pp}$ into $L_p(\Omega_0)$. The maps $J\circ S$ and $Q\circ P$ witnesses that $\ell_2$ complementably embeds into $L_{\pp}$.

Suppose a subsymmetric space $\Sym$ other than $\ell_2$ complementably embeds into $\XX:=L_{\pp}$. Set $\Omega_1=\pp^{-1}([1,2))$ and $\Omega_2=\pp^{-1}([2,\infty))$ and $\pp_j=\pp|_{\Omega_j}$, $j=1$, $2$. By Lemma~\ref{lem:DropNLC} and \cite{AnsorenaBello2025b}*{Theorem~5.2}, there is $a\in\{1,2\}$ such that $\Omega_a$ is nonnull and $\Sym$ complementably disjointly embeds into $L_{\pp_a}$. By Lemma~\ref{lem:Anso39}, $F\in\Ot(r)$ for some $r\in R(\pp_a)$.

In \cite{HR2012}*{Proposition~4.4} it is proved that $\ell_r$ complementably embeds into $L_{\pp}$ for every $r\in R(\pp_c)$. For the sake of completeness and clarity, we use our approach to variable exponent Lebesgue spaces to reprove this result. Choose $(\varepsilon_n)_{n=1}^\infty$ in $(0,\infty)$ with $\sum_{n=1}^\infty \varepsilon_n<\infty$. By Lemma~\ref{lem:PotCont}, there is a sequence $(V_n)_{n=1}^\infty$ of neighbourhoods of $q$ such that
\[
\norm{\bpsi_p-\bpsi_r}_\infty\le\varepsilon_n, \quad \norm{\bpsi_{p'}-\bpsi_{r'}}_\infty \le\varepsilon_n
\]
for all $n\in\NN$ and $p\in V_{n}$. Use Lemma~\ref{lem:disjointsets} to choose pairwise disjoint measurable sets $(\Omega_n)_{n=1}^\infty$ such that $0<\mu(\Omega_n)<\infty$ and $\pp(\Omega_n) \subseteq V_n\cap[1,\infty]$ for all $n\in\NN$. Use Lemma~\ref{lem:dualB} to pick sequences $(f_n)_{n=1}^\infty$ and $(g_n)_{n=1}^\infty$ of nonnegative measurable functions such that $\supp(f_n)=\supp(g_n)\subseteq\Omega_n$ and
\[
\int_\Omega f_n g_n \, d \mu=\rho_{\pp}^L(f_n)=\rho^L_{\pq}(g_n)=1
\]
for all $n\in\NN$. By Theorem~\ref{thm:MusEmbed}, there are bounded linear maps
\[
T\colon \ell_r \to L_{\pp}, \quad S\colon \ell_{r'} \to L_{\pp'}
\]
given by
\[
(a_n)_{n=1}^\infty \mapsto \sum_{n=1}^\infty a_n \, f_n, \quad (a_n)_{n=1}^\infty \mapsto \sum_{n=1}^\infty a_n \, g_n
\]
respectively. The map $S$ is a kernel operator whose conjugate map $S'\colon L_{\pp} \to \ell_r$ is given by
\[
S'(g)=\int_\Omega g_n(\omega) g(\omega) \, d\mu(\omega).
\]
Since $S'\circ T=\Id_{\ell_r}$, we are done.
\end{proof}

\begin{theorem}
Let $\pp$ be a variable exponent over a nonatomic $\sigma$-finite measure space $(\Omega,\Sigma,\mu)$. Suppose that $\pp^{-}>0$ and $\pp(\Omega)\subseteq(0,1]$. Given a semi-normalized sequence space $\Sym$, the following are equivalent.
\begin{itemize}
\item $\Sym$ complementably embeds into $L_{\pp}$.
\item $\Sym$ complementably disjointly embeds into $L_{\pp}$.
\item $\Sym=\ell_1$ and $\Omega_1:=\pp^{-1}(1)$ is not null.
\end{itemize}
\end{theorem}

\begin{proof}
Set $\mu_1=\mu|_{\Omega_1}$ and $\pp_1=\pp|_{\Omega_1}$. If $\Omega_1$ is nonnull, then $L{\pp_1}=L_1(\mu_1)$.

Suppose that $\Sym$ complementably embeds into $L_{\pp}$. By Lemma~\ref{lem:DropNLC}, $\Omega_1$ is nonnull and $\Sym$ complementably embeds into $L_1(\mu_1)$. By \cite{LinPel1968}, $\Sym=\ell_1$.

Conversely, if $\Omega_1$ is not null, it is well-known that $\ell_1$ complementably disjointly embeds into $L_1(\mu_1)$.
\end{proof}

\begin{lemma}\label{lem:CEDiscrete}
Suppose a subsymmetric sequence space $\Sym$ complementably embeds into an absolutely continuous sequence space $\XX$ over $\NN$. Then, $\Sym$ complementably disjointly embeds into $\XX$.
\end{lemma}

\begin{proof}
Let $\XB=(x_n)_{n=1}^\infty$ in $\XX$ with projecting functionals $\XB^*=(x_n^*)_{n=1}^\infty$ be equivalent to the unit vector system of $\Sym$. By the Cantor diagonal technique, passing to a subsequence we can assume that $\XB$ converges pointwise. Since $(x_{2n-1})_{n=1}^\infty$ is equivalent to $(x_{2n})_{n=1}^\infty$,
the mapping
\[
f \mapsto \sum_{n=1}^\infty x_{2n-1}^*(f) x_{2n-1} - \sum_{n=1}^\infty x_{2n-1}^*(f) x_{2n}
\]
defines and endomorphism of $\XX$. Consequently, we can replace $\XB$ with
\[
\enpar{x_{2n-1}-x_{2n}}_{n=1}^\infty
\]
and $\XB^*$ with
\[
\enpar{x_{2n-1}^*}_{n=1}^\infty.
\]
This way, we can assume that $\XB$ converges to zero pointwise. By the gliding-hump technique (see \cite{AnsorenaBello2025b}), passing to a subsequence we can assume that $\XB$ is finitely disjointly supported. We conclude the proof by applying Lemma~\ref{lem:disjointcoordinate}.
\end{proof}

\begin{theorem}
Let $\pp\colon\NN\to(0,\infty)$ be a variable exponent with $\pp^{-}>0$ and $\pp^{+}<\infty$. Let $\Sym$ be a subsymmetric space. The following are equivalent.
\begin{enumerate}[label=(\roman*)]
\item\label{it:discretecomp:c} $\Sym$ complementably embeds into $\ell_{\pp}$.
\item\label{it:discretecomp:a} $\Sym=\ell_r$ for some $r\in A(\pp)$.
\item\label{it:discretecomp:b} There is a subsequence $\pq$ of $\pp$ such that $\Sym=\ell_{\pq}$.
\end{enumerate}
\end{theorem}

\begin{proof}
\ref{it:discretecomp:c} implies \ref{it:discretecomp:a} by Lemma~\ref{lem:CEDiscrete} and Lemma~\ref{lem:Anso39}. In turn, \ref{it:discretecomp:a} implies \ref{it:discretecomp:b} by Lemma~\ref{thm:MusEmbed} and Lemma~\ref{lem:PotCont}. Since it is obvious that \ref{it:discretecomp:b} implies \ref{it:discretecomp:c}, we are done.
\end{proof}
\bibliography{BiblioVEL}
\bibliographystyle{plain}
\end{document}